\documentclass[a4paper]{scrartcl}
\usepackage[utf8]{inputenc}
\usepackage[T1]{fontenc}
\usepackage[english]{babel}
\usepackage{authblk}
\usepackage{amsmath, amsfonts, amssymb, amsthm, mathtools}
\usepackage{multicol}
\usepackage{enumerate}
\usepackage{calrsfs}
\usepackage{lmodern}
\usepackage{hyperref}

\DeclareMathAlphabet{\pazocal}{OMS}{zplm}{m}{n}

\newcommand{\N}{\mathbb N}

\newcommand{\D}{\displaystyle}

\newcommand{\F}{\pazocal F}

\newcommand{\K}{\pazocal K}

\newcommand{\M}{\pazocal M}
\newcommand{\dd}{\mathrm d}
\newcommand{\diam}{\mathrm{diam}}
\newcommand{\dist}{\mathrm{dist}}
\newcommand{\Lip}{\mathrm{Lip}}

\numberwithin{equation}{section}

\theoremstyle{definition}
\newtheorem{defi}{Definition}[section]

\theoremstyle{plain}
\newtheorem{thm}[defi]{Theorem}
\newtheorem{coro}[defi]{Corollary}
\newtheorem{prop}[defi]{Proposition}
\newtheorem{lem}[defi]{Lemma}
\newtheorem{exa}[defi]{Example}

\theoremstyle{remark}
\newtheorem{rem}[defi]{Remark}

\title{On Successive Approximations for Compact-Valued Nonexpansive Mappings}
\author{Emir Medjic}
\affil{\small Universit\"at Innsbruck, Department of Mathematics - 
	\texttt{emir.medjic@uibk.ac.at}}
\date{\today}

\begin{document}
\maketitle
\section*{Abstract}
We show that for a given initial point the typical, in the sense of Baire category, nonexpansive compact valued mapping $F$ has the following properties: there is a unique sequence of successive approximations and this sequence converges to a fixed point of $F$. In the case of separable Banach spaces we show that for the typical mapping there is a residual set of initial points that have a unique trajectory.\\
\textit{Keywords: }Banach space, generic property, set-valued nonexpansive mapping, successive approximations.\\ MSC2020: 47H04, 47H09, 47H10, 54E52.

\section{Introduction}
Fixed point theorems play an important role in various branches of mathematics. While Banach's fixed point theorem is valid for every (strictly) contractive self-mapping of a complete metric space, unfortunately there is no general fixed point theorem for nonexpansive mappings. Let $C$ be a closed, bounded and convex subset of a Banach space $X$. If $X$ is finite dimensional, Brouwer's fixed point theorem guarantees a fixed point for every continuous self-mapping of $C$. Since Benyamini and Sternfeld in \cite{BenyaminiSternfeld1983} showed that in infinite dimensional spaces the sphere is contractible, an analogous statement as Brouwer's cannot hold in infinite dimensional Banach spaces, even for Lipschitz-continuous mappings. For uniformly convex $X$ and for every nonexpansive self-mapping of $C$, i.e. for every mapping $f:C\to C$ with 
\[
\|f(x)-f(y)\|\leq \|x-y\|
\]
the Browder-G\"ohde-Kirk fixed point theorem implies the existence of a, possibly non-unique, fixed point. Herein lies the motivation for looking at nonexpansive mappings for fixed point theory. But also in these cases F. S. de Blasi and J. Myjak in \cite{DeBlasiMyak1976,DeBlasiMyjak1989} showed that the typical, in the sense of Baire category, nonexpansive mapping has a unique fixed point which can be approximated by iterating the mapping $f$. Since by \cite{BargetzDymondReich2017} the typical nonexpansive mapping is not a strict contraction, this behaviour cannot be explained by Banach's fixed point theorem. S. Reich and A. Zaslavski showed in \cite{ReichZaslavski2005,ReichZaslavski2008,ReichZaslavski2016} that, even in more general situations, the typical nonexpansive self mapping is contractive in the sense of Rakotch, i.e. it satisfies 
\[
\|f(x)-f(y)\|\leq \phi(\|x-y\|)\|x-y\|
\] 
for some decreasing mapping 
\[
\phi: [0,\diam(C)]\to [0,1]\text{ with }\phi(t)<1\text{ for }t>0.
\]
The mentioned fixed point theorems for single-valued mappings have one thing in common: The fixed point can be attained by iterating the mapping. This tends to be more difficult in the setting of a set-valued mapping, as it is a priori unclear what iterating such a mapping means. One way to achieve this is to take the union of the point images, which leads to lifting a compact-valued mapping to a mapping on the hyperspace of compact sets in $C$, i.e. lifting
\[
F:C\to \K(C)\text{ to a self-mapping }\tilde F: \K(C)\to \K(C),
\] which was done by Reich and Zaslavski in \cite{ReichZaslavski2002,ReichZaslavski2016-AnalysisForum,ReichZaslavski2017}. In the mentioned papers the authors have also shown that the typical compact-valued mapping is a contraction in the sense of Rakotch and has a fixed point, i.e. in the example above: there is a set $A\in \K(C)$ such that $A= \tilde F (A)$. Yet there are different approaches, for example the approach of successive approximations: starting with a point $x_0$ we take a sequence satisfying $x_{n+1}\in F(x_n)$, where we let $x_{n+1}$ be a minimizer of $\{\|x_n-y\| : y\in F(x_n)\}$. This was studied already in \cite{Pianigiani2015,BargetzReich2019}, but these papers about successive approximations of set-valued mappings focus on special mappings of the form $F=\{f,g\}$ with two nonexpansive single valued mappings $f,g:C\to C$. We show in this work that for a given initial point $x_0 \in C$ the set of mappings $F:C\to \K(C)$ with unique trajectory starting in $x_0$ is residual, improving results on successive approximations of compact-valued mappings. Moreover this sequence of successive approximations will converge to a fixed point of $F$, this can be seen in \cite{ReichZaslavski2002}.

Even though our main interest lies on mappings with compact images, a lot of our results will apply to mappings with closed and bounded point images. Set-valued mappings are of interest in various areas, for example Lipschitzian set-valued mappings are being studied in \cite{BQTA2021}. These set-valued mappings occur in a natural way, for example when computing the subdifferential of a convex mapping at a point. Recently unions of nonexpansive mappings have been considered by M. K. Tam and others; see, for example, \cite{DaoTam2019,Tam2018}. Also recently, global convergence and acceleration of fixed point iterations for union upper semicontinuous operators is being investigated by J. H. Alcantara and C. Lee; see \cite{AlcantaraLee2022}.

The properties of the typical compact subset of certain spaces have been extensively studied. In particular the question of the size of the sets of nearest and farthest points is well understood, see e.g. \cite{DeBlasi1997,DeBlasiZamfirescu1999,DeBlasiZhivkov2002}. J. Myjak and R. Rudnicki, in \cite{MyjakRyszard2001}, have shown that the typical compact set is hispid, i.e. the projection is non-unique at many points. Regarding the question how big the set of points is that have a single nearest point in an arbitrary subset of a strictly convex normed space, see \cite{Steckin1963}. In Section~\ref{sec:reverse-steckin} we provide an answer to the converse question in Theorem~\ref{thm:proj-close-sets}: Given an arbitrary point, how big is the set of closed and bounded sets that have a unique metric projection for the chosen point?

\section{Preliminaries}
Throughout this work let $X$ be a Banach space and $C\subset X$ be a non-empty, closed, convex and bounded subset of $X$. In particular $C$ does not have to be compact. When we talk about $C$ as a metric space we use the metric $\dd (x,y) = \|x-y\|$ for $x,y\in C$. For $x\in C$ and $r>0$ let $B_\dd(x,r)$ denote the open ball (and $\overline B_\dd (x,r)$ the closed ball) with radius $r$ and centre $x$ w.r.t. the metric $\dd$. If it is clear which metric is used, we will omit the subscript and use $B(x,r)$ instead of $B_\dd(x,r)$. 
\\
For $x,y \in C$  we set 
\[[x,y] : = \{ (1-\lambda)x + \lambda y ~|~ \lambda\in [0,1] \},\]
and since we assume that $C$ is convex, it is a subset of $C$. Further set $(x,y)= [x,y]\setminus\{x,y\}$ for any $x,y \in C$.
For $M\subset X$ and $x\in X$ we denote by 
\[\dist(x,M) = \inf_{y\in M}\|x-y\| \text{ and } \diam(M) = \sup_{x,y\in M}\|x-y\|,\]
the distance of $x$ to the set $M$ and the diameter of $M$ respectively.
\paragraph{Metric Hyperspace of Sets}
We consider the space of non-empty compact sets 
\begin{align}
\K(C) &:= \left \lbrace A\subset C ~\middle| ~ A \text{ is non-empty and compact}\right\rbrace \label{def:compact-sets}
\end{align}
and equip it with the Pompeiu-Hausdorff distance $H$, defined by 
\[H(A,B) := \max\left\{ \sup_{a\in A} \inf_{b\in B} \dd(a,b), \sup_{b\in B} \inf_{a\in A} \dd(a,b) \right \}.\]
It is well known that the metric space $(\K(C),H)$ is complete.

For two sets $A,B\subset X$ we set 
\[A+B = \{a+b ~ | ~ a\in A, b \in B\}.\]
\begin{lem}\label{lem:hausdorff}
	For subsets $A,B\subset X$ and $x\in X$ the inequality
	\[
	\dist(x,A) \leq \dist(x,B)+H(A,B)
	\]
	holds.
\end{lem}

\begin{proof}
	For all $a\in A$ and $b\in B$ we have by the triangle inequality 
	\[
	\dd(x,a)\leq \dd(x,b)+\dd(a,b).
	\]
	Therefore we obtain
	\[
	\inf_{a\in A} \dd(x,a) \leq \dd(x,b) + \inf_{a\in A} \dd(a,b) = \dd(x,b)+\dist(b,A).
	\]
	Since this holds for every $b\in B$, we get
	\[
	\dist(x,A)\leq \inf_{b\in B} \left( \dd(x,b) +\dist(b,A) \right) \leq \dist(x,B)+H(A,B).
	\]
\end{proof}
\paragraph{Nonexpansive and Contractive Mappings}
Let $\M$ denote the set of all nonexpansive mappings $T:C\to \K(C)$. In this sense nonexpansive means that 
\[
H(T(x),T(y)) \leq \|x-y\|\text{ for all }x,y\in C.
\]
Moreover we call mappings in $\M$ strict contractions, if there exists $L\in [0,1)$ such that
\[
	H(T(x),T(y))\leq L\|x-y\| \text{ for all }x,y\in C.
\]
We equip $\M$ with the metric of uniform convergence, i.e. 
\[\rho(F,G) := \sup_{x\in C} H(F(x), G(x)), \quad \text{for }F,G\in \M, \]
which makes $(\M,\rho)$ a complete metric space.\\
A mapping $T\in \M$ is called contractive in the sense of Rakotch, see \cite{Rakotch1962}, if there exists some decreasing function \[\phi:[0,\mathrm{diam}(C)]\to [0,1]\text{ such that }\phi(t)<1 \text{ for all }t>0,\]  with
\[
H(T(x), T(y)) \leq  \phi(\|x-y\|) \|x-y\| 
\]
for all $x,y\in C$, i.e. $\phi$ only depends on the distance between $x$ and $y$.

\paragraph{Meagre and residual}
We call a subset $E$ of a topological space $X$ meagre, if it is a countable union of nowhere dense sets. A residual set is the complement of a meagre set.
\paragraph{Porosity}
We call a subset $E$ of a metric space $X$ porous at a point $x\in E$ if there exists $r_0>0$ and $\alpha>0$ such that for all $r\in (0,r_0)$ there exists a point $y\in {B(x,r)}$ such that ${B(y,\alpha r)}\cap E = \emptyset$.\\
We call $E$ porous if it is porous at all of its points and $\sigma$-porous if it is a countable union of porous sets.\\
Whether $r_0$ and $\alpha$ are chosen independent of $x$ or not does not matter for the $\sigma$-porosity of $E$, see \cite[p. 93]{BargetzDymondReich2017}. For the sake of completeness we repeat here the proof displayed there. Take a $\sigma$-porous set $\D E = \bigcup_{i=1}^\infty E_i$, where $E_i$ are porous for all $i$. For all $j,k\in\N$ set 
\[
E_i^{j,k} := \left\{ x\in E_i \ \middle| \ r_0(x)\geq \frac{1}{j}, \alpha(x) \geq \frac{1}{k}\right\}.
\]
Then $\D E=\bigcup_{i,j,k=1}^\infty E_i^{j,k}$ and each $E_i^{j,k}$ is porous with $r_0$ and $\alpha$ independent of $x$.
\paragraph{Metric Projections}
For a set $M\subset X$ and $x\in X$ we call the set 
\[
P_M(x):=\{ y\in M \ | \ \|x-y\|= \mathrm{dist}(x,M)\}
\]
the metric projection of $x$ onto $M$.
Naturally this set can be empty, even for closed sets $M$, or contain more than one point. Since continuous mappings attain their minimum on compact sets, compact sets always have a non-empty projection.
\paragraph{Successive Approximations}
Let $F \in \M$ and $N\in \N\cup \{\infty\}$. A \textit{trajectory w.r.t.} $F$ or \textit{a sequence of successive approximations w.r.t.} $F$ is a sequence $\{x_i\}_{i=1}^N $ in $C$ with
\[ 
x_{n+1}\in P_{F(x_n)}(x_n), \quad \text{ for }  1\leq n < N.
\]
An infinite sequence of successive approximations or trajectory $\{x_n\}_{n=0}^\infty$ is called regular, if $P_{F(x_n)}(x_n)$ is a singleton for all $n$.

\paragraph{Strict Contractions are Dense in $\M$}
The proof of the following Lemma is the same as in the case of single-valued mappings, which can be found for example in \cite{DeBlasiMyak1976}. We include it for the sake of self-containedness.
\begin{lem}\label{lem:contracting-to-a-point}
	The set of strict contractions is dense in $\M$. 
\end{lem}
\begin{proof}
	Let $F\in \M, x_0\in C, \lambda\in (0,1)$ and $\varepsilon>0$. We will show that then the mapping 
	\[
	F_{\lambda,x_0}(x):= \lambda \{x_0\} + (1-\lambda) F(x)
	\]
	is a strict contraction satisfying $\rho(F_{\lambda,x_0},F)<\varepsilon$, i.e. for an arbitrary mapping $F\in\M$ we can find a strict contraction arbitrarily close to $F$. By \cite[p.119,120]{BargetzDymondReich2017} the mapping $F$ is well defined, i.e. it is a compact set for every $x\in C$. 
		
	\noindent For $x,y\in C$ and $\tilde{x}\in F(x)$ we obtain
	\begin{align*}
		&\inf \left\{\|\lambda x_0 + (1-\lambda)\tilde x - \lambda x_0 - (1-\lambda) \tilde y\| ~ \middle| ~ \tilde y \in F(y)\right\} \\
		& \hspace{3cm} =  \min \left\{(1-\lambda)\|\tilde{x}-\tilde{y}\| ~ \middle| ~ \tilde{y}\in F(y)\right\}\\
		& \hspace{3cm} = (1-\lambda) \min \left\{\|\tilde{x}-\tilde{y}\| ~ \middle| ~ \tilde{y}\in F(y)\right\}.
	\end{align*}
	The roles of $x$ and $y$ respectively are reversible in above calculation. Therefore we obtain
	\[
		H(F_{\lambda,x_0}(x),F_{\lambda,x_0}(y)) \leq (1-\lambda) H(F(x),F(y)) \leq (1-\lambda)\Lip(F) \|x-y\|.
	\]
	Further we get for $\tilde{x}\in F(x)$ that 
	\begin{align*}
		\inf &\left\{\| \tilde{x} - \lambda x_0 - (1-\lambda)y\| ~ \middle | ~ y\in F(x) \right\} \\
		&\leq \inf\left\{ \| \lambda (\tilde{x}-x_0) + (1-\lambda) (\tilde{x}-y)\| ~\middle |~ y\in F(x)\right\}\\
		&\leq \lambda \diam(C)
	\end{align*}
	Now choose $\lambda < \varepsilon/\diam(C)$ and we obtain that 
	\begin{align*}
		H(F(x),F_{\lambda,x_0}(x)) \leq \lambda \diam(C) < \varepsilon,
	\end{align*}
	which concludes the proof.
\end{proof}

\section{How many sets have a unique projection onto them for a fixed point?}\label{sec:reverse-steckin}
As mentioned in the introduction we show in some sense the converse direction to Theorem 1 in \cite{Steckin1963}. For a fixed $x\in X$ we take the complement of sets $M$ in $\K(X)$ which have a unique metric projection $P_M(x)$ and show that it is $\sigma$-porous.
The following example shows that being close to a set with unique projection does not imply that the projection is unique.
\begin{exa}
	Let $x\in X$ and $r,\varepsilon>0$ be arbitrary and set 
	\[
	M= \partial B(x,r) = \overline{B}(x,r)\setminus B(x,r).
	\]
	Further choose $y\in X$ with $r-\varepsilon < \|x-y\| <r$ and set $N = M\cup \{y\}$. Then $H(M,N)<\varepsilon$, yet $P_M(x) = M$ and $P_N(x) = \{y\}$. Moreover $\diam(P_M(x)) = 2r$.
\end{exa}
One can take various variations of the sets presented above, which leads to the question how many sets there actually are with a unique projection onto them. In the following Lemma we use a base idea similar to the idea in Lemma 1 in \cite{Steckin1963}.
\begin{lem}\label{lem:proj-close}
	Let $M\in \K(C), x\in C\setminus M$ and $y\in P_M(x)$. Then for $z\in (x,y)$ with $\dist(z,M)>0$ the set 	
	\begin{align}\label{eq:def-set-with-point}
	N := M\cup \{z\}
	\end{align}
	satisfies for every $0<\varepsilon\leq \frac{\dist(z,M)}{3}$
	\begin{align}
	P_{N'}(x) \subset \overline{B}\left(P_N(x), \varepsilon\right) = \overline B\left(z, \varepsilon\right) \text{ for all }N'\in B_H\left(N, \varepsilon\right).
	\end{align}
\end{lem}
\begin{proof}
	By construction $\overline B(z,\varepsilon)\cap M = \emptyset$ holds and let $N'\in B_H(N,\varepsilon)$. Moreover \linebreak $P_N(x)=\{z\}$ holds since $\|x-z\| < \inf_{z'\in M} \|x-z'\| $. Observe that $N'\cap \overline{B}(z,\varepsilon)\not=\emptyset$.
	For every $y'\in N'\setminus\overline{B}(z,\varepsilon)$ we have
	\[
	\|x-y'\| \geq\dist(x,M)-\varepsilon \geq \|x-z\| + 2\varepsilon.
	\]
	But for $z'\in B(z,\varepsilon)$ the following holds:
	\[ 
	\|x-z'\| < \|x-z\|+\varepsilon<\|x-y'\|. 
	\]
	From this we conclude $P_{N'}(x)\subset \overline B(z,\varepsilon)$. 
\end{proof}
\begin{thm} \label{thm:proj-close-sets}
	Let $X$ be a Banach space and $C\subset X$ a closed, bounded and convex set. Fix an arbitrary point $x\in C$. Then there is a set $\pazocal{B} \subset \K(C)$ such that $\K(C)\setminus \pazocal{B}$ is $\sigma$-porous and for every $M\in \pazocal{B}$ the set $P_M(x)$ is a singleton.
\end{thm}
\begin{proof}
	Define 
	\[\pazocal{A}_n (x) := \left\{M\in \K(C) \ | \ \diam(P_M(x))\geq \frac{1}{n}\right\},\]
	and $ \pazocal{A}(x) = \D \bigcup_{n\geq 1}\pazocal{A}_n(x)$. It is sufficient to show that $\pazocal{A}_n(x)$ is porous for every $n\geq 1$.\\
	Fix an $n\in \N$. For $M\in \pazocal{A}_n(x)$ set 
	\[
	r_0 := \min\left\{ \frac{1}{2(n+1)}, \frac{\dist(x,M)}{3}\right\} \text{ and } \alpha = \frac{1}{3}.
	\]
	For every $0<r<r_0$ choose a $z$ as in the statement of Lemma~\ref{lem:proj-close} with $\dist(z,M)=r$. Then $z$ and $N$ as in \eqref{eq:def-set-with-point} satisfy the conditions from Lemma~\ref{lem:proj-close}. In particular it follows directly that $P_N(x)=\{z\}$, since $\|x-z\| < \|x-y\|$ for all $y\in N$ by construction of $z$, and therefore $N\in \pazocal A_n(x)$. So we have for all $N'\in B_H(N,\alpha r)$
	\[
	\diam(P_{N'}(x))\leq 2\alpha r<2 r_0< \frac{1}{n}.
	\] 
	Summing up for any $M\in \pazocal A _n(x)$ there is a set $N\in \K(C)$, which is $r$-close to $M$ and fulfils $B_H(N, \alpha r) \cap \pazocal A _n(x) = \emptyset$ with $\alpha$ as set above, hence we have that $\pazocal A_n(x)$ is porous. 
\end{proof}

\section{Successive Approximations for Set-Valued Mappings}
Throughout this chapter let $F\in \M$ be a strict contraction and $\{x^F_n\}_{n\in\N} \subset C$ be a trajectory w.r.t. $F$. We will omit the $F$ in the superscript, if it is clear which mapping is meant.

For technical reasons we always assume $x_0\not\in F(x_0)$.
The aim is to prove the following claims.
\phantomsection
\label{sec:claim1}
\paragraph{Claim 1:} For $n\in \N$, $r>0$ and a strict contraction $F\in \M$, there is a mapping $G_n \in \M$ such that the trajectory starting from $x_0$ w.r.t. $G_n$ is unique in the first $n$ steps and $\rho(G_n,F)<r$.

We will achieve this by following an inductive procedure that we summarize as the following claim.
\phantomsection
\label{sec:claim2}
\paragraph{Claim 2:} Let $F,G_0, \ldots, G_{m-1}$ be already constructed and $x_0,\ldots, x_m \in C$ with
\[
P_{G_i(x_k)}(x_k)= \{x_{k+1}\} \quad \text{ for } 0\leq k\leq i\leq m-1,
\]
and $\sigma_m>0$ small enough, where the upper bound depends on the previous steps. There is a strict contraction $G_m$ and $x_{m+1}$ with $\rho(G_{m-1},G_m)<\sigma_m$ and 
\[
P_{G_m(x_k)}(x_k) = \{x_{k+1}\} \quad \text{ for } 0 \leq k\leq m.
\]

\subsection{Some Auxiliary Lemmas}
We will state a few auxiliary lemmas before working on the claims.

\begin{lem}\label{lem:Lip-on-ball}
	Let $C$ be a convex subset of a Banach space $X$, $(Y,\rho)$ a metric space, $z\in X$ and $\delta, L_1, L_2>0$. Let $f:C\to Y$ be a continuous mapping such that for all $ x,y\in B(z,\delta)\cap C$ and all $\tilde x, \tilde y \in C\setminus B(z,\delta)$ holds \[\rho(f(x),f(y))\leq L_1\|x-y\|, \quad \rho(f(\tilde x),f(\tilde y))\leq L_2 \|\tilde x- \tilde{y}\|.\]
	Then $f$ is Lipschitz-continuous with $\Lip(f)\leq \max\{L_1,L_2\}$.
\end{lem}
The above Lemma is well known and is a particular case of the much more general Lemma 2 in \cite{IvesPreiss2000}.\\

We now recall bounds on the Lipschitz constant of retractions onto balls in Banach spaces by C. F. Dunkl and K. S. Williams, \cite{DunklWilliams1964}. Similar bounds have been obtained by D. G. de Figueiredo and L. A. Karlovitz in \cite{deFigueiredoKarlovitz1967}.
Even though the retraction on the closed unit ball does not need to be nonexpansive, the following mapping $R_{\varepsilon, x_0}$ is.
\begin{lem}\label{lem:retraction}
	Let $X$ be a Banach space, $x_0\in X$ and $\varepsilon>0$. Denote by ${B}_\varepsilon $ the open ball around $0$ with radius $\varepsilon$ and by $\overline{B}_\varepsilon$ the closed ball around $0$ with radius $\varepsilon$. The mapping 
	\[
	r_\varepsilon:X \to \overline{B}_\varepsilon: x\mapsto \begin{cases}
	\varepsilon \frac{x}{\|x\|}, & \text{if }x\not\in \overline B_\varepsilon\\
	x, & \text{otherwise}
	\end{cases}
	\]
	is a retraction onto $\overline{B_\varepsilon}$, i.e. $r_\varepsilon$ is continuous and $r_\varepsilon|_{\overline B_\varepsilon}=\mathrm{id}$. Moreover the mapping 
	$$ R_{\varepsilon, x_0}: X\to X: x\mapsto x-r_\varepsilon(x-x_0)$$
	is nonexpansive.
\end{lem}
\begin{proof}
	The mapping $r_\varepsilon$ is a $2$-Lipschitz retraction by Dunkl-Williams \cite{DunklWilliams1964}. Therefore $R_{\varepsilon, x_0}$ is a continuous mapping, since $r_\varepsilon$ is Lipschitz and so in particular continuous.
	We show that $R_{\varepsilon, x_0}$ is nonexpansive by applying Lemma~\ref{lem:Lip-on-ball} to $\overline{B}(x_0,\varepsilon)$ and $X\setminus \overline{B}(x_0,\varepsilon)$.
	
	\noindent The case where $x,y\in \overline{B}(x_0,\varepsilon)$ is obvious, since here $r_\varepsilon$ is the identity and $R_{\varepsilon, x_0}$ constant with value $x_0$. Let $x,y\in X$ be any two points not in $\overline{B}(x_0,\varepsilon)$. 
	\begin{align*}
	\|R_{\varepsilon, x_0}(x)-R_{\varepsilon, x_0}(y)\| &= \left\| x-\frac{\varepsilon (x-x_0)}{\|x-x_0\|} - y + \frac{\varepsilon(y-x_0)}{\|y-x_0\|}\right\| 
	\\
	&= \left\| (x-x_0)-\frac{\varepsilon (x-x_0)}{\|x-x_0\|} - (y-x_0) + \frac{\varepsilon(y-x_0)}{\|y-x_0\|}\right\|
	\end{align*}
\begin{align*}
	\phantom{\|R_{\varepsilon, x_0}(x)-R_{\varepsilon, x_0}(y)\|} &= \left\| \left(1-\frac{\varepsilon}{\|x-x_0\|}\right)(x-x_0) -\left(1-\frac{\varepsilon}{\|x-x_0\|}\right)(y-x_0) \right. \\
	&  \qquad +\left. \left(1-\frac{\varepsilon}{\|x-x_0\|}\right)(y-x_0)- \left(1-\frac{\varepsilon}{\|y-x_0\|}\right)(y-x_0) \right\|\\
	&\leq \left(1-\frac{\varepsilon}{\|x-x_0\|}\right)\|x-y\| + \|y-x_0\|\left| \frac{\varepsilon}{\|x-x_0\|}-\frac{\varepsilon}{\|y-x_0\|}\right|\\
	&= \left(1-\frac{\varepsilon}{\|x-x_0\|}\right) \|x-y\| + \frac{\varepsilon}{\|x-x_0\|} |\|x-x_0\|-\|y-x_0\||\\
	&\leq \|x-y\|.
	\end{align*}
\end{proof}
\begin{rem}\label{rem:retraction}
	The above lemma can be interpreted in the following way: The mapping $R_{\varepsilon, x_0}$ adds a small vector to $x$ in direction of $x_0$, i.e. $R_{\varepsilon, x_0}(x) = x+\xi$ with $\xi = -r_\varepsilon(x-x_0)$ and $\|\xi\| \leq \varepsilon$. In particular for $x\in B_\varepsilon(x_0)$ the mapping is constant with value $x_0$. 
\end{rem}
\begin{rem}
	A somewhat similar way of perturbing was independently developed by M. Dymond, in \cite{dymond2021}.
\end{rem}
The following Lemma is well-known but for the convenience of the reader we include it with a proof.
\begin{lem}\label{lem:hausdorff-dist-ineq}
	Let $A_1,A_2,B_1,B_2\in \K(C)$. Then 
	\[H(A_1\cup A_2, B_1) \leq \max\{H(A_1,B_1), H(A_2,B_1)\}\] 
	and 
	\[ H(A_1\cup A_2, B_1\cup B_2) \leq \max\{H(A_1,B_1), H(A_2,B_2)\}.\]
\end{lem}
\begin{proof}
	Set $\delta := \max\{H(A_1,B_1), H(A_2,B_2)\}$ and for $A\in \K(C)$ and $\varepsilon>0$ denote by $A^\varepsilon$ the set $A+B(0,\varepsilon)$.
	Then we have by definition of $\delta$ that $A_1\subset B_1^ \delta$, $A_2\subset B_2^\delta$ and consequently
	\[A_1\cup A_2 \subset B^\delta_1\cup B^\delta_2=(B_1\cup B_2)^\delta \quad \text{ and } \quad B_1\cup B_2 \subset (A_1\cup A_2)^\delta.\]
	Therefore the last inequality stated in the Lemma is true and the first one follows as a special case.
\end{proof}

\subsection{Mappings with Isolated Fixed Points and Without Fixed Points}

In this part of the section we show that for a fixed $x_0\in C$ the typical mapping in $\M$ has no fixed point in a ball around $x_0$. Moreover we show first that every mapping with a fixed point has a mapping close to it with the same fixed point being an isolated point.

\begin{lem}\label{lem:define-L}
	Let $F\in \M$, $x_0,z\in C$ and $\delta >0$ such that $F(x_0)\cap \overline B(z,\delta/2)\not=\emptyset$. Define the mapping $\mu_{x_0,\delta}:C\cap\overline{B}(x_0,\delta)\to [0,1]$ by
	\[
		\mu_{x_0,\delta}(x):= \frac{4}{3\delta} \max\left\{0,\|x-x_0\|-\frac{\delta}{4}\right\}.
	\]
	Then the mapping $L_{x_0,z,\delta}: C\to \K(C)$ defined by
	\[
	 x\mapsto \begin{cases}
		\left(1-\mu_{x_0,\delta}(x)\right)\{z\} + \mu_{x_0,\delta}(x) (F(x_0)\cap \overline B(z,\delta/2)),& x\in C\cap B(x_0,\delta),\\
		F(x_0)\cap \overline B(z,\delta/2), & x \in C\setminus B(x_0,\delta),
	\end{cases}
	\]
	is Lipschitz continuous with constant $\Lip(L_{x_0,z,\delta})=\frac{2}{3}$.
\end{lem}
\begin{proof}
	Note that $(F(x_0)\cap \overline{B}(z,\delta/2))$ and $\{z\}$ are compact sets, and therefore their sum is compact too. Hence $L_{x_0,z,\delta}$ is well defined. In \cite[p.119, 120]{BargetzDymondReich2017} the authors show for $t,s\in [0,1]$ and $A\in \K(C)$ with $A\not=\{z\}$ that
	\begin{align}\label{eq:Hausdorff-sets-contraction}
		H(t\{z\} + (1-t)A,s\{z\} + (1-s)A) = |s-t|H(\{z\},A)
	\end{align}
	holds. Observe that $\mu_{x_0,\delta}$ is a continuous function that is constant on $C\cap B(x_0,\delta/4)$ and has Lipschitz constant $\frac{4}{3\delta}$ on $C\cap B(x_0,\delta)\setminus B(x_0,\delta/4)$ and therefore is Lipschitz with $\Lip(\mu_{x_0,\delta})=\frac{4}{3\delta}$ by Lemma~\ref{lem:Lip-on-ball}. In \eqref{eq:Hausdorff-sets-contraction} plugging in $x,y\in \overline{B}(x_0,\delta)$, $t= \mu_{x_0,\delta}(x)$, \linebreak$ s=\mu_{x_0,\delta}(y)$ and $A=F(x_0)\cap \overline{B}(z,\delta/2)$ yields
	\begin{align*}
		H(L_{x_0,z,\delta}(x),L_{x_0,z,\delta}(y)) &= |\mu_{x_0,\delta}(x)-\mu_{x_0,\delta}(y)|H(\{z\},F(x_0)\cap \overline{B}(z,\delta/2))\\ 
		&\leq \Lip(\mu_{x_0,\delta})\|x-y\|\frac{\delta}{2}\\
		&\leq \frac{2}{3}\|x-y\|.
	\end{align*}
	It is easy to see that $\mu_{x_0,\delta}(x)=1$ for $x$ with $\|x-x_0\|=\delta$, and therefore $L_{x_0,z,\delta}(x)= F(x_0)\cap\overline{B}(z,\delta/2)$. So the mapping $L_{x_0,z,\delta}$ is continuous, constant on $C\setminus B(x_0,\delta)$ and $C\cap B(x_0,\delta/4)$, has Lipschitz constant $2/3$ on $C\cap B(x_0,\delta)\setminus B(x_0,\delta/4)$. Therefore by applying Lemma~\ref{lem:Lip-on-ball} twice the proof is concluded.
\end{proof}
\begin{lem}\label{lem:make-fixed-point-isolated}
	Let $F\in \M$ and $x_0\in C$ be such that $x_0\in F(x_0)$. Then for every \linebreak $\varepsilon\in (0,\diam(C))$ and for $0<r \leq \varepsilon/24$ there is a mapping $G\in B_\rho(F,\varepsilon)$ such that $\Lip(G)\leq \max\{2/3 , \Lip(F)\}$, $x_0\in G(x_0)$ and for all $v\in B(x_0,r)$ we have $P_{G(v)}(v)=\{x_0\}$ and $G(v)\cap B(x_0,r)=\{x_0\}$.
\end{lem}
\begin{proof}
	Set $\varepsilon'=\varepsilon/2, \delta = \varepsilon/4$ and define the mapping $G:C\to \K(C)$ by
	\[
	G(x):= \begin{cases}
		F(R_{\varepsilon'\negthickspace,x_0}(x)),& x\in C\setminus B(x_0,\varepsilon'),\\
		(F(x_0)\setminus {B}(x_0,\delta/2))\cup L_{x_0,x_0,\delta}(x), & \text{otherwise}
	\end{cases} ,
	\]
	where $L_{x_0,x_0,\delta}$ is as in Lemma~\ref{lem:define-L}. We use that $F(R_{\varepsilon'\negthickspace,x_0}(x)) = F(x_0)$ for $x\in B(x_0,\varepsilon')\cap C$.	We will show that $\Lip(G)\leq\max\{2/3, \Lip(F)\}$. Let first $x,y\in B(x_0,\varepsilon')\cap C$. Then we have that
	\begin{align}\label{eq:Lip-GL} \nonumber
		H(G(x),G(y)) 
		&= H((F(x_0)\setminus {B}(x_0,\delta/2)) \cup L_{x_0,x_0,\delta}(x), (F(x_0)\setminus {B}(x_0,\delta/2)) \cup L_{x_0,x_0,\delta}(y))\\ 
		&\leq H(L_{x_0,x_0,\delta}(x),L_{x_0,x_0,\delta}(y))\\ \nonumber
		&\leq \frac{2}{3} \|x-y\|,
	\end{align}
	where we applied Lemma~\ref{lem:hausdorff-dist-ineq} in the penultimate inequality and Lemma~\ref{lem:define-L} in the last inequality. Further for $x,y\in C\setminus B(x,\varepsilon')$ we have
	\begin{align}\label{eq:Lip-GF-2}
		H(G(x),G(y))=H(F(R_{\varepsilon'\negthickspace,x_0}(x)),F(R_{\varepsilon'\negthickspace,x_0}(y))) \leq \Lip(F) \|x-y\|,
	\end{align}
	since $\Lip(R_\varepsilon'\negthickspace,x_0)=1$ by Lemma~\ref{lem:retraction}.\\
	Note that for $x$ with $\|x-x_0\| = \varepsilon'$ we have $	L_{x_0,x_0,\delta}(x) = F(x_0)\cap \overline{B}(x_0,\delta/2) $	and
	\[
	F(x_0)=(F(x_0)\setminus B(x_0,\delta/2)) \cup F(x_0)\cap \overline{B}(x_0,\delta/2) = F(R_{\varepsilon'\negthickspace,x_0}(x)),
	\]
	hence $G$ is continuous and by the above observations together with Lemma~\ref{lem:Lip-on-ball} we have \linebreak $\Lip(G)\leq\max\{2/3, \Lip(F)\}$. Now set $r = \delta/6$ and observe that for all $v\in B(x_0,r)$ we have
	\[
		G(v) = G(x_0) \quad \text{and} \quad P_{G(v)}(v) = \{x_0\},
	\]
	since $L_{x_0,x_0,\delta}$ is constant on $B(x_0,r)$ by definition and the second part follows by a straight forward calculation. It remains to show that $\rho(F,G)<\varepsilon$. Observe that for $x\in C\setminus B(x_0,\varepsilon')$ we have
	\begin{align}\label{eq:dist-GL-F-1}
		H(G(x),F(x)) &= H(F(R_{\varepsilon'\negthickspace,x_0}(x)),F(x)) \\ \nonumber &\leq \Lip(F)\|x-r_\varepsilon'(x-x_0)-x\| \\ \nonumber &\leq \Lip(F)\varepsilon' \leq \varepsilon/2.
	\end{align}
	For $x\in B(x_0,\varepsilon')\cap C$ we have that
	\begin{align} \label{eq:dist-GL-F-2}
		H(G(x),F(x)) &= H((F(x_0)\setminus {B}(x_0,\delta/2))\cup L_{x_0,x_0,\delta}(x),F(x)) \\ \nonumber
		&\leq H((F(x_0)\setminus{B}(x_0,\delta/2)) \cup L_{x_0,x_0,\delta}(x),F(x_0))+ H(F(x_0),F(x)) \\ \nonumber
		&< \max\left\{H(F(x_0)\setminus B(x_0,\delta/2), F(x_0)\setminus B(x_0,\delta/2)),\right. \\ \nonumber &\hspace{5cm} \left. H(L_{x_0,x_0,\delta}(x), F(x_0)\cap \overline{B}(x_0,\delta/2))\right\} + \varepsilon'\\ \nonumber
		&\leq \delta + \varepsilon' < \varepsilon,
	\end{align}
	where we used the triangle inequality, split up the set $F(x_0)$ in 
	\[
		F(x_0) = (F(x_0)\setminus {B}(x_0,\delta/2)) \cup 	(F(x_0)\cap 	\overline{B}(x_0,\delta/2)),
	\]
	a union of two sets, in order to apply Lemma~\ref{lem:hausdorff-dist-ineq} and then used that \linebreak $L_{x_0,x_0,\delta}(x)\subset \overline{B}(x_0,\delta/2)$, therefore $H(L_{x_0,x_0,\delta}(x),\overline{B}(x_0,\delta/2))\leq \delta$, in the penultimate inequality.
\end{proof}
\begin{prop}\label{prop:TypicalFixedPoint}
	Let $x_0\in C, F\in \M$ and $\diam(C)>\varepsilon >0$. Then there are a mapping $G\in \M$ and a radius $r>0$ such that $\rho(F,G)<\varepsilon$ and for all $v\in B(x_0,r)$ we have $v\not \in G(v)$.
\end{prop}
\begin{proof}
	First we investigate the case where $x_0\not\in F(x_0)$. Since $F(x_0)$ is a compact set we have a positive distance $s=\dist(x_0,F(x_0))>0$. Observe now that for all $v\in B(x_0,s/3)$ we have
	\begin{align}
		\nonumber s=\dist(x_0,F(x_0))&\leq \|x_0-v\| + \dist(v,F(x_0)) \\ \label{eq:dist-fixed-point}
		& \leq \frac{s}{3} + \dist(v,F(v)) + H(F(v),F(x_0)) \leq \frac{2}{3}s + \dist(v,F(v))
	\end{align}
	and hence the inequality $\dist(v,F(v))\geq s/3$ holds. So one can take $F=G$ and $r=s/3$.
	
	\noindent Now assume $x_0\in F(x_0)$. Set $\varepsilon' = \varepsilon/2$ and $\delta=\varepsilon/4$. We use the mapping $R_{\varepsilon'\negthickspace,x_0}(x) = x- r_\varepsilon'(x-x_0)$ from Lemma~\ref{lem:retraction}. Choose an arbitrary $z\in C$ with $\|z-x_0\|=\delta/4$.
	Define the mapping
	\[
	G(x):= \begin{cases}
		F(R_{\varepsilon'\negthickspace,x_0}(x)),& x\in C\setminus B(x_0,\varepsilon'),\\
		(F(R_{\varepsilon'\negthickspace,x_0}(x))\setminus {B}(z,\delta/2))\cup L_{x_0,z,\delta}(x), & \text{otherwise}
	\end{cases} .
	\]	
	Note that for $x\in \partial B(x_0,\varepsilon')$ we have $G(x) = F(R_{\varepsilon'\negthickspace,x_0}(x))=F(x_0)$ and therefore this mapping is continuous.
	By construction we have that $\dist(x_0,G(x_0))\geq\delta/4$ and therefore we can choose $r:=\delta/12$ to guarantee that no $v\in B(x_0,r)$ is a fixed point of $G$ by the same argument as in \eqref{eq:dist-fixed-point}. We still have to show that $\Lip(G)\leq 1$ and $\rho(F,G)<\varepsilon$ holds. The first of the claims follows by the same reasoning as in \eqref{eq:Lip-GL} and \eqref{eq:Lip-GF-2}. For the latter one we reason as in \eqref{eq:dist-GL-F-1} and \eqref{eq:dist-GL-F-2} again to obtain $\rho(F,G)<\varepsilon$. 
\end{proof}
\begin{coro}\label{coro:A0-nowheredense}
	Let $x_0\in C$ and define the set of mappings in $\M$ that do not have $x_0$ as fixed point as
	\[
	\pazocal A_0(x_0) := \left\{ F\in \M ~\middle|~ x_0\not\in F(x_0) \right\}.
	\]
	Then the set $\M\setminus \pazocal A_0(x_0)$ is nowhere dense.
\end{coro}
\begin{proof}
	In view of Proposition~\ref{prop:TypicalFixedPoint} we are left to show that there is a ball around $F\in \pazocal A_0(x_0)$ fully contained in $\pazocal A_0(x_0)$. We set $s=\dist(x_0,F(x_0))$ and observe that every $G\in B_\rho(F,s/3)$ has to satisfy $H(G(x_0),F(x_0))<s/3$. Therefore $x_0\not \in G(x_0)$.
\end{proof}

\subsection{The Inductive Construction} \label{sec:inductive-constr}

Given a strict contraction we will construct a new strict contraction, such that the new mapping has a unique projection from a given point and is still close to the mapping we started from. In the following we fix $n\in \N$ and $0<r<1$.

\begin{lem} \label{lem:lambda} 
	For $x_0\in C$ and $\delta>0$ define 
	\[\lambda_{x_0,\delta}: C \to [0,1]: x\mapsto \frac{\max\{\tfrac{\delta}{2}-\|x-x_0\|,0\}}{\delta / 2} .\]
	Then $\Lip(\lambda_{x_0,\delta}) \leq \frac{2}{\delta}$ and $\lambda_{x_0,\delta}(x)=0$ for $x\in C\setminus B(x_0,\delta/2)$.
\end{lem}
\begin{proof}
	For $x,y\in C\setminus {B}(x_0,\delta/2)$ we have $|\lambda_{x_0,\delta}(x)-\lambda_{x_0,\delta}(y)|=0$. For $x,y\in \overline{B}(x_0,\delta/2)\cap C$ we have that \[|\lambda_{x_0,\delta}(x)-\lambda_{x_0,\delta}(y)|=\frac{|\delta/2-\|x-x_0\| - \delta/2 +\|y-x_0\||}{\delta/2} \leq \frac{2}{\delta} \|x-y\|.\]
	The mapping $\lambda_{x_0,\delta}$ is continuous as a composition of continuous mappings and the claim follows by Lemma~\ref{lem:Lip-on-ball}.
\end{proof}

\begin{prop}\label{prop:construction-g}
	Let $G\in \M$, $\Lip(G)<1$, $z\in C$ with $z\not\in G(z)$, $\delta := \dist(z,G(z))$, $\sigma\in (0,\delta/2)$ and $\varepsilon':=(1-\Lip(G))\sigma/4$. Then there exists a mapping $\tilde{G}\in \M$ and a point $\tilde{z}\in C\setminus G(z)$ that satisfy:
	\begin{multicols}{2}
		\begin{enumerate}[(i)]
			\item $\tilde G(z) = G(z)\cup \{\tilde z\}$ \label{prop:constr:form}
			\item $P_{\tilde G(z)}(z) = \{\tilde z\}$\label{prop:constr:proj}
			\item $\tilde{z}\not \in \overline{B}(z,\sigma)$\label{prop:constr:sigma}
			\item $\Lip(\tilde G) \leq \Lip(G)/3 + 2/3 < 1$ \label{prop:constr:Lip}
			\item $\rho(G, \tilde G) < \sigma$ \label{prop:constr:distance}
			\item for all $ \tilde x \in B(z,\frac{\varepsilon'}{3}) : P_{\tilde G(z)}(\tilde x) = \{\tilde z\}$.\label{prop:constr:points-close}
		\end{enumerate} 
	\end{multicols}
\end{prop}
We will prove that proposition in smaller parts, but want to show the bigger picture first. Inserting $F$ for the mapping $G$ and $x_0$ for $z$ in Proposition~\ref{prop:construction-g} will yield the mapping $G_0$ and $x_1$, so we get
\begin{align*}
&G=F ,\quad \ G=G_0 ,\quad   &&\ldots,  &&&G=G_{n-1}\\
&\quad \downarrow \qquad \qquad \downarrow && \quad &&& \downarrow \qquad \ \  \\
&\tilde G=G_0 ,\quad  \tilde G = G_1, \quad  &&\ldots,  &&& \tilde G = G_n \ \ \
\end{align*}
and a sequence of points $x_0,\ldots, x_{n+1}$. Using suitable conditions on the constants, this sequence satisfies the conditions for the inductive construction in \hyperref[sec:claim2]{Claim 2}.

\paragraph{Proof of Proposition~\ref{prop:construction-g}} We divide the proof into a number of lemmas. Let $G$ and $z$ be as in Proposition~\ref{prop:construction-g}. Let $y\in P_{G(z)}(z)$ be arbitrary and set
\begin{align} \label{eq:const-0}
\delta := \dist(z,G(z)),\ 0<\sigma \leq \frac{\delta}{2}, \ \varepsilon := \frac{\sigma (1-\Lip(G))}{3}, \ \varepsilon' := \frac{3}{4}\varepsilon.
\end{align}
Moreover let $\lambda_{z,\sigma}$ be as in Lemma~\ref{lem:lambda} and 
\[
\tilde z := \frac{\varepsilon'}{\delta} z + (1-\frac{\varepsilon'}{\delta})y.
\]
Thus, 
\[\dist(\tilde z, G(z))=\|\tilde z-y\| = \left\|\frac{\varepsilon'}{\delta} (z-y) + y-y\right\| = \frac{\varepsilon'}{\delta} \|z-y\| = \varepsilon'. \]
We define
\begin{align} \label{eq:constr-g}
g=g_{\tilde z,\sigma}:C\to \K(C):x\mapsto \{\tilde z\lambda_{z,\sigma}(x) + y(1-\lambda_{z,\sigma}(x))\},
\end{align}
and note that for all $x\not\in B(z,\sigma/2)$ the mapping is constant with $g(x)=y$. Note that the Lipschitz constant of $g$ is bounded from above by $2 \frac{\varepsilon'}{\sigma}$. To see this, let $x_1,x_2\in C$ and note that
\begin{align}
H(g(x_1),g(x_2)) &= \|\tilde z\lambda_{z,\sigma}(x_1)+(1-\lambda_{z,\sigma}(x_1))y - \tilde z\lambda_{z,\sigma}(x_2) - (1- \lambda_{z,\sigma}(x_2))y\| \nonumber \\
&\leq \|\tilde z(\lambda_{z,\sigma}(x_1)-\lambda_{z,\sigma}(x_2)) +y(\lambda_{z,\sigma}(x_2)- \lambda_{z,\sigma}(x_1))\|\nonumber  \\ &= |\lambda_{z,\sigma}(x_1)-\lambda_{z,\sigma}(x_2)|\cdot \|y-\tilde z\|\nonumber \\
&\leq \frac{2\varepsilon}{\sigma}\|x_1-x_2\|. \label{eq:lipschitz-g}
\end{align}
Further recall $R_{\sigma,z}$ from Lemma~\ref{lem:retraction} and define $\tilde G: C\to \K(C)$ by 
\begin{align} \label{eq:constr-G}
\tilde G(x) := \begin{cases} 
G(R_{\sigma,z}(x)), & x\not \in B(z, \sigma) \\
G(R_{\sigma,z}(x))\cup g(x),& x\in \overline{B}(z,\sigma)
\end{cases}.
\end{align}
Taking into account Lemma~\ref{lem:proj-close}, it is clear that $P_{\tilde G(z)}(z) = \{\tilde z\}$. So $\tilde G$ satisfies $ (\ref{prop:constr:form}) $ and $ (\ref{prop:constr:proj}) $. Moreover the upper bound on $\sigma$ ensures that $\tilde z\not \in \overline{B}(z,\sigma)$, i.e. $ (\ref{prop:constr:sigma}) $. The next lemma shows claim $(\ref{prop:constr:Lip})$.
\begin{lem}\label{lem:G-Lipschitz}
	The mapping $\tilde G$ satisfies $\Lip(\tilde G)\leq \max\{\Lip(G),2\varepsilon/\sigma\}\leq  1$.  In particular since $G$ is a strict contraction, so is $\tilde G$. 
\end{lem}
\begin{proof}
	First observe that for all $x \in C$ with $\|x-z\|=\sigma$ we have 
	\[
	\tilde{G}(R_{\sigma,z}(x)) = G(z) = G(R_{\sigma,z}(x))\cup g(x),
	\]
	since $g(x) \subset G(z)$ for all $x$ with distance $\sigma$ to $z$. Therefore the mapping $\tilde{G}$ is continuous.\\
	Using \eqref{eq:lipschitz-g}, Lemma~\ref{lem:retraction}, Lemma~\ref*{lem:hausdorff-dist-ineq} and Lemma~\ref{lem:lambda} for $x_1,x_2\in C\cap B(z,\sigma)$ we obtain
	\begin{align*}
	H(\tilde G(x_1),\tilde G(x_2)) &= H(G(R_{\sigma, z}(x_1))\cup g(x_1), G(R_{\sigma, z}(x_2))\cup g(x_2)) \\
	&\leq H(G(R_{\sigma, z}(x_1))\cup g(x_1), G(R_{\sigma, z}(x_2))\cup g(x_1)) + \\
	&\qquad \qquad +H(G(R_{\sigma, z}(x_2))\cup g(x_1), G(R_{\sigma, z}(x_2))\cup g(x_2))\\
	&\leq \max\{\Lip(G)H(R_{\sigma, z}(x_1),R_{\sigma, z}(x_2)), H(g(x_1), g(x_1))\} + \\
	&\qquad \qquad + \max\{H(G(R_{\sigma, z}(x_2)),G(R_{\sigma, z}(x_2))),H(g(x_1),g(x_2))\}\\
	&\underset{\text{by \eqref{eq:lipschitz-g}}}{\leq} \frac{2{\varepsilon}}{\sigma} \|x_1-x_2\|. 
	\end{align*}
	where we used that $R_{\sigma,z}(x_1)=z=R_{\sigma,z}(x_2)$.
	In the case of $x_1,x_2\in C\setminus B(z,\sigma)$ one obtains
	\begin{align*}
	H(\tilde G(x_1),\tilde G(x_2)) &= H(G(R_{\sigma, z}(x_1)), G(R_{\sigma, z}(x_2))) \\ 
	&\leq \Lip(G) \|R_{\sigma, z}(x_1)-R_{\sigma, z}(x_2)\| \\
	&\leq \Lip(G)\|x_1-x_2\|.
	\end{align*}
	Now Lemma~\ref{lem:Lip-on-ball} and the choice of $\varepsilon$ yield
	\[\Lip(\tilde G) \leq \max\left\{\Lip(G), \frac{2\varepsilon}{\sigma}\right\} =\max\left\{ \Lip(G) , \frac{2 (1-\Lip(G))}{3}\right\} \leq 1,\]
	where equality would only be attained when $\Lip(G)=1$ and therefore this yields the claimed result.
\end{proof}
The next lemma shows $(\ref{prop:constr:distance})$.
\begin{lem}\label{lem:f-g-close}
	$G$ and $\tilde G$ satisfy $\rho(G,\tilde G)<\sigma.$
\end{lem}
\begin{proof}
	In the case of $x \in C\setminus B(z,{\sigma})$ we may use Lemma~\ref{lem:retraction} to obtain
	\[H(G(x),\tilde G(x)) = H(G(x),G(R_{\sigma, z}(x)))\leq \Lip(G)\underbrace{\|x-R_{\sigma, z}(x)\|}_{\mathclap{=\left\|\sigma \frac{x-z}{\|x-z\|}\right\|}} \leq \Lip(G){\sigma}.\]
	Taking $x\in C\cap B(z,{\sigma})$ and using Lemma~\ref{lem:hausdorff-dist-ineq} together with the triangle inequality, we obtain that
	\begin{align*}
	H(G(x),\tilde G(x)) &= H(G(x),G(R_{\sigma, z}(x))\cup g(x))\\
	&\leq H(G(x),G(z)) + H(G(z),G(z)\cup g(x))\\
	&\leq \Lip(G)\sigma + \varepsilon'
	\end{align*}
	This yields
	\[H(G(x),\tilde G(x))\leq {\sigma}\left(\Lip(G) + \frac{1-\Lip(G)}{4}\right)= \frac{1+3\Lip(G)}{4}\sigma<{\sigma}.\]
\end{proof}

The way $\tilde G$ is constructed yields a very useful property. Not only the projection from $z$ onto $\tilde G(z)$ is unique, but every point close to $z$ is projected on $\tilde z$, i.e. $(\ref{prop:constr:points-close})$ holds:
\begin{lem}\label{lem:constr-proj-ball-x0}
	For every $\tilde x\in C$ with $	\|\tilde x - z \| < \varepsilon'/3$ the identity $P_{\tilde G(z)}(\tilde x) = \{\tilde z\}$ holds.
\end{lem}
\begin{proof}
	Let $\tilde{x}\in C\cap B(z, {\varepsilon'}/3)$ and assume there would be a $\tilde{y}\in \tilde G(z)\setminus\{\tilde z\}$ such that $\tilde y \in P_{\tilde G(z)}(\tilde x)$. Now $\tilde z \in \tilde G(z)$ implies $\|\tilde{x}-\tilde{y}\|\leq \|\tilde{x}-\tilde z\|$. \\
	Moreover $\tilde{y}\in G(z)$, since $\tilde y\not = \tilde z$ and 
	\[\tilde G(z) = G(R_{\sigma,z}(z))\cup g(z) = G(z)\cup\{\tilde z\}.\]
	Thus  
	\[\|z- \tilde{y}\| \leq \|z - \tilde{x}\| + \|\tilde{x}-\tilde{y}\| < \frac{{\varepsilon'}}{3} + \|\tilde{x}-\tilde z\| < \frac{2{\varepsilon'}}{3} + \underbrace{\|z-\tilde z\|}_{=\delta - \varepsilon'} < \delta,\]
	which is a contradiction to $\delta = \dist(z,G(z))$.
\end{proof}
\noindent This concludes the proof of Proposition~\ref{prop:construction-g}. 

\paragraph{Proofs of Claim 1 and Claim 2}
For the purpose of proving \hyperref[sec:claim1]{Claim 1} stated above, we will need further restrictions on $\sigma$ and set 
\begin{align} \label{eq:def-sigma-0}
\sigma_0:=r\min\left\{\frac{\delta_0}{2}, \frac{1}{n+1}\right\},
\end{align}
and
$\delta_0, \varepsilon_0, \varepsilon'_0$ as in \eqref{eq:const-0}. Let $F$, $G_0$, $x_0$ and $x_1$ be $G$, $\tilde G$, $z$ and $\tilde z$ as in \eqref{eq:constr-G} or construct $G_0$ as in Lemma~\ref{lem:make-fixed-point-isolated} if $x_0\in F(x_0)$, i.e. $\{x_1\}= P_{G_0(x_0)}(x_0)$. Assume that we have $F,G_0, G_1,\ldots G_{m-1}$ and $x_0,\ldots, x_m \in C$ already constructed, with constants $\delta_{i-1}, \sigma_{i-1}, \varepsilon_{i-1}, \varepsilon'_{i-1}$ and 
\begin{align} \label{eq:inductive-traj}
P_{G_{i-1}(x_{k})}(x_{k})=\{x_{k+1}\} \quad \text{ for } 0\leq k< i\leq m,
\end{align}
as in \hyperref[sec:claim2]{Claim 2}. We are presented with two possibilities. Either $x_m\not\in G_{m-1}(x_m)$ or $x_m$ is a fixed point of $G_{m-1}$ and is the first in the sequence with this property. In the first case we set
\begin{align} \label{eq:const-i} 
\delta_m &:= \dist(x_m,G_{m-1}(x_m)), \\ \sigma_m&:=\min\left\{\frac{\varepsilon'_{m-1}}{3}, \frac{\delta_m}{2}, \frac{\|x_m-x_0\|}{2},\ldots, \frac{\|x_m-x_{m-1}\|}{2} \right\}, \label{eq:const-i-2} \\ \varepsilon_m &:= \frac{\sigma_m(1-\Lip(G_{m-1}))}{3},\quad  \varepsilon'_m := \frac{\sigma_m(1-\Lip(G_{m-1}))}{4}.\label{eq:const-i-3}
\end{align}
In particular the $\sigma_i$ satisfy $0< \sigma_i < r/(n+1)$. In the case of $\delta_m = 0$, that is when $x_m \in G_{m-1}(x_m)$, we use Lemma~\ref{lem:make-fixed-point-isolated} to construct a $G_m$ such that $\rho(G_m, G_{m-1})<\sigma_m$ and $x_m\in G_m(x_m)$ becomes an isolated point. Then we set all the following $G_{m+j} = G_m$ and $x_{m+j}=x_m$ for $j\geq 1$. Further we continue the sequences $\sigma_m,\varepsilon_m,\varepsilon'_{m}$ by setting 
\begin{align}\label{eq:const-i-fixed-point-1}
	\sigma_{m}&:= \varepsilon'_{m-1}/3,\\ \label{eq:const-i-fixed-point-2}
	\varepsilon_{m} &:= \min\left\{ \frac{\sigma_{m}}{18}, \frac{\sigma_{m}(1-\Lip(G_{m}))}{3} \right\} \text{ and }\\ \label{eq:const-i-fixed-point-3}
	\varepsilon'_{m}&:= \min\left\{ \frac{\sigma_{m}}{24}, \frac{\sigma_{m}(1-\Lip(G_{m}))}{4} \right\}.
\end{align}
The number $\sigma_m$ as in \eqref{eq:const-i-2} only vanishes when $\delta_m$ vanishes, this is a consequence of the following Lemma.
\begin{lem}\label{lem:never-repeating-trajectory}
	For a strict contraction $G$, $\zeta_0\in C$, $n\in \N$ and the trajectory $\{\zeta_i\}_{i=0}^\infty$ starting at $\zeta_0$ w.r.t. $G$ we have that all $\zeta_i$ are different as long as none of them is a fixed point, i.e. $\zeta_i=\zeta_j$ if and only if $i=j$ for $0\leq i,j\leq m$ or $\zeta_i \in G(\zeta_i)$. Moreover, for $1\leq i\leq m$,
	\[
	\dist(\zeta_i,G(\zeta_i)) \leq \Lip(G)^i \dist(\zeta_0,G(\zeta_0)).
	\]
\end{lem} 
\begin{proof}
	First note that
	\begin{align*}
	\|\zeta_{i+1}-\zeta_i\|= \dist(\zeta_i,G(\zeta_i)) \leq& \quad \dist(\zeta_i,G(\zeta_{i-1})) + H(G(\zeta_{i-1}),G(\zeta_i))\\
	\underset{\mathclap{\zeta_i\in G(\zeta_{i-1})}}{\leq}& \quad \Lip(G) \|\zeta_i-\zeta_{i-1}\|.
	\end{align*}
	The first inequality is Lemma~\ref{lem:hausdorff}. Applying this inductively yields the desired inequality. Now, assume w.l.o.g. $i>j$, we observe that
	\begin{align*}
	\|\zeta_i-\zeta_j\| \geq \|\zeta_j-\zeta_{j+1}\|-\sum_{k=j+1}^{i-1}\|\zeta_k-\zeta_{k+1}\|\geq \|\zeta_j-\zeta_{j+1}\| \left( 1-\sum_{k=1}^{i-j-1}\Lip(G)^{k}  \right) \geq 0,
	\end{align*}
	and equality only holds when $\zeta_j = \zeta_{j+1}$, i.e. when $\zeta_j$ is already a fixed point. If none of the $\zeta_i$ is a fixed point we have $\|\zeta_i-\zeta_j\|>0$ for $i\not = j$.
\end{proof}
\noindent In the case that $x_m\not\in G_{m-1}(x_m)$, now using Proposition~\ref{prop:construction-g} with $G=G_{m-1}$, $z=x_m$, $\delta=\delta_{m}$, $\sigma=\sigma_{m}$, $\varepsilon' = \varepsilon'_m$ and \eqref{eq:constr-G} we construct the mapping $\tilde G = G_m$ and $\tilde z = x_{m+1}$. \\
In the case that $x_m \in G_{m-1}(x_m)$ we distinguish two cases: If $m$ is the first index that we reach a fixed point, we use Lemma~\ref{lem:make-fixed-point-isolated} to construct $G_m$ such that $x_{m+1}=x_m$ is an isolated point in $G_m(x_m)$. If $m$ is not the first index that we reach a fixed point, we set $G_m=G_{m-1}$. Further set $\sigma_m, \varepsilon_m, \varepsilon'_m$ as in \eqref{eq:const-i-fixed-point-1}, \eqref{eq:const-i-fixed-point-2} and \eqref{eq:const-i-fixed-point-3}. 
We show now that $G_m$ for $m\leq n$ is $r$-close to $F$ and so one part of \hyperref[sec:claim1]{Claim 1} holds true for $m=n$ and the choice of $\sigma_i$ as in \eqref{eq:const-i-2} for $0\leq i\leq m$.
\begin{lem}\label{lem:distance-to-F}
	The mappings $G_0,\ldots, G_m$ satisfy $\rho(G_i,F)<r$ for $0\leq i\leq m\leq n$.
\end{lem}
\begin{proof}
	Since $\rho(G_0,F)<\sigma_0$, $\rho(G_i, G_{i-1}) < \sigma_i$ for $1\leq i\leq m$ and $\sigma_i \leq \frac{r}{n+1}$ for $0\leq i\leq m$, the triangle inequality yields 
	\[\rho(G_m, F) \leq \rho(G_m, G_{m-1}) + \cdots + \rho(G_i,G_{i-1}) + \cdots + \rho(G_0,F) <\sum_{i=0}^m \sigma_i \leq r\sum_{i=0}^n \frac{1}{n+1}=r.\vspace*{-0.5cm}\]
\end{proof}

It is helpful at this point to rephrase Proposition~\ref{prop:construction-g}$(\ref{prop:constr:points-close})$ in terms of $G_0,G_1,\ldots G_m$.
\begin{lem}\label{lem:constr-proj-ball-inductive}
	Let $0\leq i< m$ and $\tilde x\in B(x_i,\varepsilon'_{i}/3)$. Then 
	\[
		P_{G_i(x_i)}(\tilde x) = \{x_{i+1}\}.
	\]
	Further, for $j<i$, a consequence of above statement is that for $\tilde x \in \overline{B}(x_j, \varepsilon'_i/3)$
	\[
		P_{G_i(x_j)}(\tilde x) = \{x_{j+1}\} 
	\] 
	holds.
\end{lem}
\begin{proof}
	In the case that none of the $x_i$ are fixed points of $G_m$ the first part is a direct application of Proposition~\ref{prop:construction-g}$(\ref{prop:constr:points-close})$ to $G_i$ and $x_i$. The second part follows since $\varepsilon'_i < \varepsilon'_j$, i.e.
	\[
	\overline B(x_j, \varepsilon'_i/3)\subset \overline B(x_j, \varepsilon'_j/3)
	\]
	and again Proposition~\ref{prop:construction-g}$(\ref{prop:constr:points-close})$.
	In the case that one $x_i$ is a fixed point of $G_m$, and therefore every following $x_{i+1},x_{i+2},...$ point too, by construction $x_i$ is an isolated point of $G_m(x_i)$ by Lemma~\ref{lem:make-fixed-point-isolated}. Therefore, since $\varepsilon_m' \leq \sigma_m/2$, Lemma~\ref{lem:make-fixed-point-isolated} implies $\varepsilon_m'$ defined in \eqref{eq:const-i-fixed-point-3} is small enough to satisfy above claims.
\end{proof}

With this preparation we can prove the following proposition, which is the major tool to prove \hyperref[sec:claim1]{Claim 1} and \hyperref[sec:claim2]{Claim 2}. The proposition and the lemmas before show how to construct the mappings $G_0,\ldots, G_m$, but the following proposition shows that we can indeed continue inductively.
\begin{prop}\label{prop:trajectory-unchanged}
	Let the mappings $F,G_0, \ldots, G_m$ and the trajectory $x_0,\ldots, x_m,x_{m+1}$ be constructed as before. Then the following holds:
	\begin{enumerate}[(i)]
		\item $R_{\sigma_{k+1},x_{k+1}}(R_{\sigma_{k+2},x_{k+2}}(\ldots R_{\sigma_{j+1},x_{j+1}}(R_{\sigma_{j},x_j}(x_k))\ldots )) \in B(x_k, \sigma_k/2)$ for all $0\leq k< j\leq m$,
		\item $G_j(x_k) = G_k(x_k)$ for all $0 \leq k\leq j\leq m$,
		\item $P_{G_j(x_k)}(x_k) = \{x_{k+1}\}$ for all $0\leq k\leq j\leq m$. \label{prop:part3:trajectory-unchanged}
		\item $G_j(\tilde{x})=G_k(x_k)$ and  $P_{G_j(\tilde{x})}(\tilde{x})=P_{G_k(x_k)}(x_k)$ for all $0\leq k\leq j\leq m$ and for all $\tilde{x}\in B(x_k,\varepsilon'_j/3)$.\label{prop:part4:stability}
	\end{enumerate}
\end{prop}

\begin{proof}
	As in Remark~\ref{rem:retraction} we set for $k+1\leq l \leq j$ 
	\[
	\xi_l=R_{\sigma_l,x_l}(x_k) -x_k = -r_{\sigma_l}(x_k-x_l).
	\] 
	Note that $\|\xi_l\|\leq \sigma_l < \sigma_k$. This yields
	\begin{align} \label{eq:rewriting-freeze}
	R_{\sigma_{k+1},x_{k+1}}(R_{\sigma_{k+2},x_{k+2}}(\ldots R_{\sigma_{j},x_j}(x_k))) = x_k + \sum_{l=k+1}^{j}\xi_l
	\end{align}
	and by using $\|\xi_l\|\leq \sigma_l \leq \sigma_{l-1}/3$, we further obtain
	\[
		\left\| \sum_{l=k+1}^{j}\xi_l \right\| < \sum_{l=k+1}^{j}\sigma_l \leq \frac{\sigma_k}{3}\sum_{l=0}^{j-k-1} 3^{-l} = \frac{\sigma_k}{2}\left( 1-\frac{1}{3^{j-k}} \right)<\frac{\sigma_k}{2}. 
	\]
	This means that $R_{\sigma_{k+1},x_{k+1}}(R_{\sigma_{k+2},x_{k+2}}(\ldots R_{\sigma_{j},x_j}(x_k))) \in {B}(x_k,\sigma_k/2)$ and this shows $(i)$.
	Assume now that none of the $x_k$for $0\leq k\leq m$ are a fixed point of $G_m$. Recall that $G_j$ and $\sigma_j$ are defined as in \eqref{eq:constr-G} and \eqref{eq:const-i-2}. In particular, because 
	\[
		\frac{\|x_k-x_l\|}{2} \geq \sigma_j, \quad \text{ for } k+1\leq l \leq j,
	\]
	we cannot be in the second case of \eqref{eq:constr-G}. So we always have 
	\begin{align}\label{eq:rewriting-G-j-to-G-k}
	G_j(x_k) = G_k(R_{\sigma_{k+1},x_{k+1}}(R_{\sigma_{k+2},x_{k+2}}(\ldots R_{\sigma_{j},x_j}(x_k)))).
	\end{align}
	Also we get therefore 
	\begin{align}\label{eq:G-k-distance-x-k}
		G_k(R_{\sigma_{k+1},x_{k+1}}(R_{\sigma_{k+2},x_{k+2}}(\ldots R_{\sigma_{j},x_j}(x_k))))=G_k\left(x_k + \sum_{l=k+1}^{j}\xi_l\right) = G_k(x_k)
	\end{align}
	by \eqref{eq:constr-G}. Assume now that from some $0\leq i<m$ the sequence becomes constant, i.e. $x_i = x_{i+1} = \cdots =x_m$ and $x_i\in G_m(x_i)$ is a fixed point. Then the above still holds by the same argument until $G_{i-1}$. This argument actually also holds for $G_i$, since $\|x_{i-1}-x_i\|/2\geq \sigma_i$ and we only change $G_{i-1}$ globally by less than $\sigma_i$, see Lemma~\ref{lem:make-fixed-point-isolated}. The rest follows by definition of $G_j$, since $G_i = G_{i+1}=\cdots = G_m$.
	This shows $(ii)$ and $(iii)$. Let now $\tilde{x}\in B(x_k,\varepsilon_j'/3)$. If $x_k$ is a fixed point, by the choice of $\varepsilon_j'$ we have that $(iv)$ holds as a consequence of Lemma~\ref{lem:make-fixed-point-isolated}. Now assume that $x_k$ is not a fixed point. Using $\varepsilon'_j<\sigma_k/2$ yields that if we replace $x_k$ by $\tilde{x}$ in \eqref{eq:rewriting-freeze} we obtain that 
	\[
		\left\| \tilde{x}-x + \sum_{l=k+1}^{j}\xi_l \right\|\leq \|\tilde{x}-x\| + \sigma_k/2 \leq \varepsilon_j'/3 + \sigma_k/2 < \sigma_k.
	\] 
	This yields that \eqref{eq:G-k-distance-x-k} holds also for $\tilde{x}$ instead of $x_k$. Thus rewriting $G_j(\tilde{x})$ as in \eqref{eq:rewriting-G-j-to-G-k} and using \eqref{eq:G-k-distance-x-k} we obtain
	\begin{align} \label{eq:G-j-z-equal-G-k-x-k}
		G_j(\tilde{x})=G_k(\tilde{x})=G_k(x_k)
	\end{align}
	for $\tilde{x}\in B(x_k,\varepsilon'_j/3)$. To conclude $(\ref{prop:part4:stability})$ we use \eqref{eq:G-j-z-equal-G-k-x-k} and  insert it in Lemma~\ref{lem:constr-proj-ball-inductive}.
\end{proof}

Thus the construction of $G_0,\ldots, G_n$ for $m=n$ yields that the first $n$ steps of the trajectory starting at $x_0$ w.r.t. $G_n$ are $x_0,\ldots, x_n$, which shows the final piece of \hyperref[sec:claim1]{Claim~1}. Furthermore Proposition~\ref{prop:trajectory-unchanged}$(\ref{prop:part3:trajectory-unchanged})$ yields the last missing part of \hyperref[sec:claim2]{Claim 2}.

\section{Mappings with Non-Regular Trajectories Form a Meagre Subset}
Consider, for $x_0\in C$, the sets
\begin{align*}
	\pazocal A_n(x_0) :=\Big\{F\in \M  \ \Big| \ \exists R>0 \text{ s.t. } &\diam(P_{F(v^F_k)}(v^F_k))\leq  \beta_n \text{ for all } v_0^F,\ldots, v_n^F \text{ with }\\  & \hspace*{1.2cm} v_0\in B(x_0,R), \ v^F_k\in P_{F(v^F_{k-1})}(v^F_{k-1}), n\geq k\geq 1\Big\},
\end{align*}
where $\beta_n\xrightarrow{n\to \infty} 0$ is a decreasing null sequence. We omit $F$ in the superscript, if it is clear to which mapping the trajectory refers. Observe also that 
$$
\bigcap_{n\in \N}\pazocal{A}_n(x_0) = \{F\in \M~|~ \{x_k\}_{k=0}^\infty \text{ is a regular trajectory w.r.t. }F\}.
$$
We will now show that the sets $\M\setminus \pazocal{A}_n(x_0)$ are nowhere dense.\\

\begin{prop}\label{prop:nowhere-dense}
	Let $x_0\in C$ and $n\in \N$. The set $\M\setminus \pazocal A_n(x_0)$ is nowhere dense.
\end{prop}
\begin{proof}
	Assume that w.l.o.g. $\beta_n \leq 1$, otherwise take $(\beta'_n)_{n\in\N}$ with $\beta'_n = \beta_n / \sup_{n\in\N} \beta_n$ instead of $\beta_n$. Let $\min\{1,\diam(C)\}>r'>0$, $r=r'/2$ and $F\in \M\setminus \pazocal A_n(x_0)$ be given. By Lemma~\ref{lem:contracting-to-a-point} we can find a strict contraction $F'$ with $\rho(F,F')<r$.
	We now use \hyperref[sec:claim1]{Claim 1} from the previous section for ${r}, F'$ and $x_0$ to construct $G_0,\ldots, G_n$ together with constants $(\delta_i,\sigma_i, \varepsilon_i, \varepsilon'_i)_{i=0}^n$ as in \eqref{eq:const-i}, \eqref{eq:const-i-2} and \eqref{eq:const-i-3} or \eqref{eq:const-i-fixed-point-1}, \eqref{eq:const-i-fixed-point-2} and \eqref{eq:const-i-fixed-point-3} respectively, depending on whether the trajectory reaches a fixed point or not. By Lemma~\ref{lem:distance-to-F} the mapping $G_n$ is $r$ close to ${F'}$ and hence $r'$ close to $F$.
	
	\noindent Set 
	\[
	R := \frac{\beta_n\varepsilon_n'r}{6}
	\]
	and observe that by Proposition~\ref{prop:trajectory-unchanged}$(\ref{prop:part4:stability})$ for all $v_0\in B(x_0,R)$
	\[
		G_n(v_0)= G_n(x_0)\quad \text{and} \quad P_{G_n(v_0)}(v_0) = \{x_1\} = P_{G_n(x_0)}(x_0)
	\]
	 holds.
	
	\noindent Thus this $R$ is suited for the one in the definition of the sets $\pazocal A_n$ and we will show now that $G_n$ is in $\pazocal A_n(x_0)$. Lemma~\ref{lem:G-Lipschitz} guarantees that $G_n\in \M$ and  Lemma~\ref{lem:constr-proj-ball-inductive} implies that the trajectory w.r.t. $G_n$ and starting with $v_0$ is unique in the first $n$ steps, i.e. $G_n\in \pazocal{A}_n(x_0)$.
	
	\noindent Set 
	\[
	s := \frac{\beta_n \varepsilon'_n r}{6}
	\] and let $G'\in B_\rho(G_n,s)$. 
	We will show that $G'\in \pazocal{A}_n(x_0)$. Note that 
	\[
		H(G'(v_0),G_n(v_0))\leq \rho(G',G_n) < s <  \frac{\beta_n\varepsilon'_n}{6}.
	\] 
	First we will show that each point in a trajectory $\{v_0,z_i\}_{i=1}^n$ starting at $v_0$ w.r.t. $G'$ in the first $n$ steps will be $\varepsilon_n'/6$ close to points in the unique trajectory $\{v_0,x_i\}_{i=1}^n$ w.r.t. $G_n$.\\
	The set $G_n(v_0) = G_0(v_0)$ has a special form satisfying 
	\[
		G_n(v_0)\cap B\left(x_1,\frac{\varepsilon_0'}{3}\right) = \{x_1\}.
	\]
	Since $\varepsilon'_n < \varepsilon'_0$ and $H(G'(v_0),G_n(v_0))<\beta_n \varepsilon_n'/6$ we can apply Lemma~\ref{lem:proj-close} and obtain that the projection of $v_0$ onto $G'(v_0)$ has to be contained in a ball around $x_1$, i.e.
	\[
		P_{G'(v_0)}(v_0) \subset \overline{B}(x_1,\beta_n\varepsilon'_n/6).
	\]
	Thus we have that
	\[
		\diam(P_{G'(v_0)}(v_0))\leq \frac{\beta_n\varepsilon'_n}{3}<\beta_n.
	\] 
	
	\noindent Take an arbitrary point $z_1\in P_{G'(v_0)}(v_0)$, hence 
	\begin{align}\label{eq:x_1-z_1}
		\|z_1-x_1\|<\frac{\beta_n\varepsilon_n'}{6},
	\end{align}
	and observe that 
	\[
		H(P_{G_n(x_1)}(x_1),P_{G'(z_1)}(z_1))\leq H(P_{G_n(x_1)}(x_1),P_{G_n(z_1)}(z_1)) + H(P_{G_n(z_1)}(z_1),P_{G'(z_1)}(z_1)).
	\]
	Using \eqref{eq:x_1-z_1} we can apply Proposition~\ref{prop:trajectory-unchanged}$(\ref{prop:part4:stability})$ and obtain $P_{G_n(x_1)}(x_1)=P_{G_n(z_1)}(z_1)$, hence 
	\[
		H(P_{G_n(x_1)}(x_1),P_{G_n(z_1)}(z_1)) = 0. 
	\]
	Further, since $H(G_n(z_1),G'(z_1))<\beta_n\varepsilon_n'/6$ and since Proposition~\ref{prop:trajectory-unchanged} $(\ref{prop:part4:stability})$ yields \linebreak $G_n(z_1)=G_n(x_1)$, we can apply Lemma~\ref{lem:proj-close} and obtain that 
	\[
		P_{G'(z_1)}(z_1)\subset \overline{B}\left(x_2,\frac{\beta_n\varepsilon_n'}{6}\right).
	\]
	We obtain in conclusion 
	\[
		H(P_{G_n(x_1)}(x_1),P_{G'(z_1)}(z_1)) = 		H(\{x_2\},P_{G'(z_1)}(z_1)) <\beta_n\varepsilon_n'/6.
	\]
	
	\noindent Assume now that we have $k<n$, a sequence $x_0,z_1,\ldots,z_{k-1}$ of successive approximations w.r.t. $G'$ and that $\|z_{k-1}-x_{k-1}\|<\beta_n\varepsilon'_n/6$ holds.\\
	We again obtain
	\begin{align}\label{eq:x_k-to-PGz_k}
		H(\{x_{k}\},P_{G'(z_k)}(z_k)) < \frac{\beta_n\varepsilon_n'}{6}
	\end{align}
	by the same means of applying Lemmas~\ref{lem:proj-close} and \ref{lem:constr-proj-ball-inductive} and the induction hypothesis $\|x_{k-1}-z_{k-1}\|<\beta_n\varepsilon_n'/6$. The equation \eqref{eq:x_k-to-PGz_k} again yields
	\[
		\|x_k-z_k\| < \frac{\beta_n\varepsilon_n'}{6}
	\]
	and therefore also
	\[
		P_{G'(z_{k-1})}(z_{k-1})\subset \overline{B}\left(x_k,\frac{\beta_n\varepsilon_n'}{6}\right) \Rightarrow \diam \left( P_{G'(z_{k-1})}(z_{k-1})\right)<\frac{\beta_n\varepsilon_n'}{3}.
	\]
	Hence for $\{v_0,z_k\}_{k=1}^n$, an initial segment of a trajectory w.r.t. $G'$, we have 
	\[\diam(P_{G'(z_k)}(z_k)) \leq\frac{\beta_n\varepsilon'_n}{3} \quad \text{ for } k=1,\ldots,n.\]
	
	Thus $G'\in \pazocal{A}_n(x_0)$ and $B_\rho(G_n,s)\subset \pazocal{A}_n(x_0)$.
\end{proof}

\begin{rem}
	It does not matter if $x_0$ has a unique projection onto $F(x_0)$ or not. What matters is that $G(x_0)$ creates an isolated point, in order to gain control over the trajectory. Therefore this step is done in any case.
\end{rem}

In the construction of the mappings $G_i$ we need $F$ to be a strict contraction, so the $G_i$ in the construction maintain a Lipschitz constant less than one. This is in particular needed for the trajectories to not repeat unless they are at a fixed point already, see Lemma~\ref{lem:never-repeating-trajectory}.

\noindent As a direct consequence of Proposition~\ref{prop:nowhere-dense} we obtain the following theorem.
\begin{thm}\label{thm:main}
	The sets $\pazocal{A}_n(x_0)$ contain an open and dense subset of $\M$. Moreover, the set
	\begin{align*}
	\pazocal{A}(x_0) =\bigcap_{n\in\N} \pazocal{A}_n (x_0) 
	\end{align*}
	is residual.
\end{thm}

\begin{coro}
	For every $x_0\in C$, there is a residual subset $\F\subset \M$ such that all $F\in \F$ admit a regular trajectory starting with $x_0$, which converges to a fixed point of $F$.
\end{coro}
\begin{proof}
	For the convergence to a fixed point see Theorems 4.2, 4.3 and 4.4 in \cite{ReichZaslavski2002}.
\end{proof}

\begin{thm}
	Let $X$ be a separable Banach space. Then there is a residual set $\F\subset \M$ such that for every $F\in \F$ there is a residual subset $\pazocal U \subset C$ with the following property:
	For every $x_0\in \pazocal U$ the mapping $F$ admits a regular trajectory starting at $x_0$.
\end{thm}
\begin{proof}
	Since $X$ is separable we can find a countable dense subset $D\subset C$. We use for $n\in \N$ the sets $\pazocal A_n(y_i)$ from Theorem~\ref{thm:main}, set 
	\[
		\pazocal A_n := \bigcap_{y\in D} \pazocal A_n(y) \quad \text{and} \quad \pazocal F := \bigcap_{i=0}^{\infty} \pazocal A_i
	\]
	and notice that $\pazocal A_n$ and $\F$ are residual subsets of $\M$. Fix $F\in \F$ and let $n \in \N$. Then $F\in \pazocal A_n$ and hence also $F\in \pazocal A_n(y)$ for all $y\in D$. In particular this means that there exists an $R_{n,y}$, see the definition of the sets $\pazocal A_n$, such that $\diam\left(P_{F(v_k)}(v_k)\right)\leq \beta_n$ for all $v_0,\ldots, v_n$ with $v_0\in B(y,R_{n,y})$ and $v_{k+1}\in P_{F(v_k)}(v_k)$. 

	Define for $n\geq 1 $ the sets 
	\[
		\pazocal U_n := \bigcup_{y\in D} B\left( y, R_{n,y} \right) = \left\lbrace x\in C~ \middle|~ \|x-y\|<R_{n,y}  \text{ for a } y\in D \right\rbrace .
	\]
	Then the sets $\pazocal U_n$ are dense in $C$, as $D$ is dense itself. These sets consist of all points such that the trajectory starting at one of those points behaves well in the first $n$ steps. Further this yields that the set 
	\[
		\pazocal U = \bigcap_{i=1}^\infty \pazocal U_i
	\]
	is a residual subset of $C$. By the observations above for every $x_0\in \pazocal U$ the mapping $F$ has a regular trajectory starting at $x_0$.
\end{proof}

\paragraph{Acknowledgement.} This research was funded by the Austrian Science Fund (FWF): P 32523. I wish to thank Christian Bargetz for a lot of fruitful discussions, for reading this article carefully and his helpful comments. Furthermore I want to give special thanks to the referee for carefully reading this paper, pointing out mistakes and improvements and making valuable suggestions.
This version of the article has been accepted for publication after peer review but is not the Version of Record and does not reflect post-acceptance improvements, or any corrections.
\paragraph{Data Availability. }Data sharing not applicable to this article as no datasets were generated or analysed during the current study.

\vspace{8mm}
\noindent
Emir Medjic\\
Universit\"at Innsbruck\\
Department of Mathematics\\
Technikerstraße 13,\\
6020 Innsbruck,
Austria\\
\texttt{emir.medjic@uibk.ac.at}

\end{document}